\documentclass{amsart}

\usepackage{url}
\usepackage{xcolor}
\definecolor{darkgreen}{rgb}{0,0.5,0}

\usepackage[T1]{fontenc}
\usepackage{amsmath,amsfonts,amssymb,amsthm}
\usepackage{mathtools}
\usepackage{comment}
\usepackage{pgf,tikz}
\usetikzlibrary{matrix,arrows,calc}
\usepackage{tikz-cd}
\usepackage{enumitem}
\usepackage[colorinlistoftodos]{todonotes}

\usepackage[
	backend=bibtex,
	style=alphabetic,
	minalphanames=3,
	maxnames=99,
	maxalphanames=4,
	giveninits=true,
	isbn=false,
	doi=false,
]{biblatex}
\usepackage[
colorlinks, citecolor=darkgreen,
]{hyperref}
\usepackage{cleveref}

\numberwithin{equation}{section}

\newtheorem{thmABC}{Theorem}

\newtheorem{thm}{Theorem}

\newtheorem{lemma}[thm]{Lemma}
\newtheorem{conjecture}[thm]{Conjecture}

\theoremstyle{remark}
\newtheorem{rem}[thm]{Remark}

\newtheorem{example}[thm]{Example}

\theoremstyle{definition}
\newtheorem{defn}[thm]{Definition}

\numberwithin{thm}{section}

\newcommand{\BK}{\mathrm{BK}}
\newcommand{\bC}{\mathbb C}

\newcommand{\bG}{\mathbb G}
\newcommand{\glob}{\mathrm{glob}}
\newcommand{\loc}{\mathrm{loc}}

\newcommand{\rH}{\mathrm H}
\newcommand{\HS}{\mathrm{HS}}

\newcommand{\bN}{\mathbb N}

\newcommand{\bP}{\mathbb P}

\newcommand{\bQ}{\mathbb Q}
\newcommand{\Qbar}{{\overline{\mathbb Q}}}
\newcommand{\bZ}{\mathbb Z}
\newcommand{\cO}{\mathcal{O}}
\newcommand{\cX}{\mathcal{X}}
\newcommand{\cY}{\mathcal{Y}}
\newcommand{\bR}{\mathbb{R}}
\newcommand{\fS}{\mathfrak{S}}

\newcommand{\one}{\mathbf{1}}

\newcommand{\Bcris}{\mathrm{B_{cris}}}

\newcommand{\lto}{\longrightarrow}

\DeclareMathOperator{\ab}{{\mathrm{ab}}}

\DeclareMathOperator{\cD}{\mathcal{D}}

\DeclareMathOperator{\bF}{\mathbb{F}}
\DeclareMathOperator{\Gal}{Gal}

\DeclareMathOperator{\gr}{gr}

\DeclareMathOperator{\res}{res}
\DeclareMathOperator{\Jac}{Jac}
\DeclareMathOperator{\Li}{Li}

\DeclareMathOperator{\NS}{NS}

\DeclareMathOperator{\Hom}{Hom}

\DeclareMathOperator{\rank}{\mathrm{rk}}

\DeclareMathOperator{\Sel}{Sel}
\DeclareMathOperator{\Spec}{Spec}

\DeclareMathOperator*{\colim}{co{\lim}}

\DeclareFontFamily{U}{wncy}{}
\DeclareFontShape{U}{wncy}{m}{n}{<->wncyr10}{}
\DeclareSymbolFont{mcy}{U}{wncy}{m}{n}
\DeclareMathSymbol{\Sha}{\mathord}{mcy}{"58}

\title{Linear and Quadratic Chabauty for affine hyperbolic curves}
\author{Marius Leonhardt}
\address{Marius Leonhardt,
	Mathematisches Institut,
	Ruprecht-Karls-Universität Heidelberg,
	Im Neuenheimer Feld 205,
	69120 Heidelberg,
	Germany}
\email{mleonhardt@mathi.uni-heidelberg.de}

\author{Martin Lüdtke}
\address{Martin Lüdtke,
  Bernoulli Institute, 
  University of Groningen,
  Nijenborgh 9,
  9747 AG Groningen,
  The Netherlands
}
\email{m.w.ludtke@rug.nl}

\author{J. Steffen M\"uller}
\address{ J. Steffen M\"uller,
  Bernoulli Institute, 
  University of Groningen,
  Nijenborgh 9,
  9747 AG Groningen,
  The Netherlands
}
\email{steffen.muller@rug.nl}

\begin{document}

\thispagestyle{empty}

\begin{abstract} 
	We give sufficient conditions for finiteness of linear and
        quadratic refined Chabauty--Kim loci 
        of affine hyperbolic curves. 
	We achieve this by constructing depth $\leq 2$ quotients of the
        fundamental group, following a construction of Balakrishnan--Dogra in the projective case. 
	We also apply Betts' machinery of weight filtrations to give
        unconditional explicit upper bounds on the number of $S$-integral points 
        when our conditions are satisfied.
\end{abstract}

\maketitle

\section{Introduction}
\label{sec: introduction}

Let $Y/\bQ$ be a smooth affine hyperbolic curve and let $\cY/\bZ_S$ be a
regular model of $Y$ over the ring of $S$-integers for some finite set of primes~$S$. By the theorems of Siegel and Faltings, the set of $S$-integral points $\cY(\bZ_S)$ is finite. 
However, this result is in general not effective.
One approach towards effectivity is the method of
Chabauty--Coleman~\cite{Cha41, Col85} and its nonabelian generalisation due
to
Minhyong Kim~\cite{Kim05, Kim09}, by which $\cY(\bZ_S)$ is regarded as a subset of the $p$-adic integral points $\cY(\bZ_p)$ for some prime~$p \not\in S$ of good reduction, and $p$-adic analytic functions on $\cY(\bZ_p)$ are constructed that vanish on $\cY(\bZ_S)$. More precisely, we are interested in the refined Chabauty--Kim method, as developed by Betts--Dogra \cite{BD:refined}, which produces a descending sequence of subsets
\[ \cY(\bZ_p) \supseteq \cY(\bZ_p)_{S,1} \supseteq \cY(\bZ_p)_{S,2} \supseteq \ldots, \]
all containing $\cY(\bZ_S)$. We call these the \emph{refined
Chabauty--Kim loci}.\footnote{The refined Chabauty--Kim loci
$\cY(\bZ_p)_{S,n}$ are denoted $\cY(\bZ_p)_{S,n}^{\min}$ in
\cite{BD:refined} to distinguish them from the non-refined, possibly larger
Chabauty--Kim loci. In this paper we only consider the refined variant,
therefore we omit the superscript $(-)^{\min}$ from the notation.} It is
conjectured that $\cY(\bZ_p)_{S,n}$ is finite for sufficiently large $n \gg
0$, in which case it is given as the vanishing set of 
nontrivial Coleman functions. (In fact, we expect $\cY(\bZ_p)_{S,n}$ to be
equal to $\cY(\bZ_S)$ for sufficiently large~$n$; this is the refined
version of \emph{Kim's conjecture} \cite[Conjecture~3.1
]{BDCKW}.) Thus, computing the set $\cY(\bZ_p)_{S,n}$, whenever it is
finite, can serve as an approximation to computing the set of $S$-integral points.

In general, the refined Chabauty--Kim loci $\cY(\bZ_p)_{S,n}$ are 
difficult to compute. 
Significant progress has only been made in the cases $n = 1$ and $n = 2$,
which correspond to \emph{linear} and \emph{quadratic Chabauty},
respectively. 
In this paper, we give sufficient criteria on $(Y,S,p)$ for
finiteness of 
$\cY(\bZ_p)_{S,1}$ (Theorem \ref{thm: alpha}\ref{thm: depth 1 finiteness}) and
$\cY(\bZ_p)_{S,2}$ (Theorems \ref{thm: alpha}\ref{thm: depth 2 finiteness} and \ref{full weight -2 finiteness}).
In addition, 
we obtain bounds on the size of the quadratic Chabauty--Kim locus $\cY(\bZ_p)_{S,2}$, which also bound $\#\cY(\bZ_S)$. Our results (Theorems \ref{thm: beta} and \ref{full weight -2
coleman function}) in this direction have the form that whenever a certain
inequality holds, then $\#\cY(\bZ_p)_{S,2}$ is bounded in terms of  invariants
associated to $(\cY, S, p)$ that can often be computed
explicitly.

Our theorems are affine analogues of the following
results for the set of rational points on \emph{projective} hyperbolic
curves: The classical
theorem of Chabauty~\cite{Cha41} proves finiteness 
whenever the rank-genus inequality $g-r > 0$ is satisfied. More recently,
Balakrishnan and Dogra showed finiteness
whenever the inequality $g -r + \rho - 1 >0$ involving the Picard
number~$\rho$ of the Jacobian is satisfied by developing and applying 
quadratic Chabauty (see~\cite[Lemma~3.2]{BD18}, strengthened
in~\cite[Proposition~2.2]{BD-II}). They also proved
an effective version~\cite[Theorem~1.1]{BD19} giving a bound on the number of rational points. 
(For different approaches to quadratic Chabauty, see~\cite{EL21}
and~\cite{BMS21}.)

Previous finiteness results for Chabauty--Kim in depth $\le 2$ for affine hyperbolic
curves of genus $>0$ were restricted to $S=\varnothing$ (see~\cite{Kim10}
and~\cite[Remark~3.3]{BD18}), and bounds for such curves are
only known for $S=\varnothing$ and $Y$ hyperelliptic
(see~\cite[Theorem~1.3]{BD19}).

We illustrate   our results with several special cases and examples in Section~\ref{sec: examples}: 

\begin{itemize}
	\item the rank equals genus case (\Cref{rem: rank equals genus});
	\item totally ramified superelliptic curves (\Cref{ex:hyp-odd});
	\item even degree hyperelliptic curves (\Cref{ex:hyp-even});
	\item the thrice-punctured line (\Cref{ex: thrice-punctured line}).
\end{itemize}

In order to state our main results precisely, we introduce some notation.
Suppose that the smooth affine hyperbolic curve~$Y/\bQ$ is given as $X
\setminus D$ where $X/\bQ$ is a smooth projective curve and $D \neq \varnothing$
is the reduced boundary divisor. We call the points in $D$ \emph{cusps} or \emph{points at infinity}. Let~$\cX$ be a regular model of $X$ \cite[Definition~10.1.1]{liu2006algebraic} over the ring~$\bZ_S$ of $S$-integers. Let $\cD$ be the closure of~$D$ in~$\cX$ and set $\cY \coloneqq \cX \setminus \cD$. Assume that $\cY$ admits an $S$-integral point.
Fix a prime~$p \not\in S$ 
such that $\cX_{\bF_p}$ is smooth and $\cD_{\bF_p}$ is étale. We use the following notation, which will be kept throughout this paper:
\begin{itemize}
	\item $r \coloneqq \rank \Jac_X(\bQ)$ the Mordell--Weil rank of the Jacobian of~$X$;
	\item $r_p \coloneqq \rank_{\bZ_p} \Sel_{p^\infty}(\Jac_X)$ the $p^\infty$-Selmer rank of the Jacobian of~$X$;
	\item $g$ the genus of~$X$;
	\item $\rho \coloneqq \rank \NS(\Jac_X)$ the Picard number of the Jacobian of~$X$; note that $\rho \geq 1$ if $g\geq 1$;
	\item $\rho_f \coloneqq \rho + \rank \NS(\Jac_{X_{\Qbar}})^{\sigma=-1}$, where $\sigma$ denotes complex conjugation and $A^{\sigma=\pm 1}$ denotes the $\pm 1$-eigenspace of a $\langle \sigma \rangle$-module $A$;
	\item $\#|D| > 0$ the number of closed points at infinity;
	\item $n \coloneqq \#D(\Qbar)>0$ the number of geometric points at infinity;
	\item write $n = n_1 + 2n_2$ with $n_1 \coloneqq \#D(\bR)$ the number of real points and $n_2$ the number of conjugate pairs of complex points at infinity;
	\item $b \coloneqq \#|D| +n_2-1$;
        \item $s\coloneqq \#S$.
\end{itemize}

The condition that~$Y$ is hyperbolic is equivalent to $2 - 2g - n < 0$.
Note also that $D \neq \varnothing$ and thus $b \geq 0$ since we are
assuming~$Y$ to be affine. Our first finiteness theorem 
for the linear and quadratic Chabauty--Kim loci now 
reads as follows. 

\begin{thmABC}\label{thm: alpha}
	\leavevmode
	\begin{enumerate}[label=(\arabic*)]
		\item\label{thm: depth 1 finiteness} 
		If 
		\[ \alpha_1(Y, s, p) \coloneqq g - r_p + b - s > 0, \]
		then $\cY(\bZ_p)_{S,1}$ is finite.
		\item\label{thm: depth 2 finiteness}
		If 
		\[ \alpha_2(Y, s, p) \coloneqq \alpha_1(Y,s,p) + \rho_f > 0, \]
		then $\cY(\bZ_p)_{S,2}$ is finite.
	\end{enumerate}
\end{thmABC}

If one assumes the Tate--Shafarevich conjecture, the $p^\infty$-Selmer
rank~$r_p$ appearing in $\alpha_1(Y,s,p)$ and $\alpha_2(Y,s,p)$ can be replaced with the Mordell--Weil rank~$r$. There is also an unconditional variant of \Cref{thm: alpha} using~$r$ instead of~$r_p$; see \Cref{rem: r versus rp} below.

Once we have finiteness of the Chabauty--Kim loci, we know that they are  defined by
one or more Coleman functions. Under suitable assumptions it is
possible to get some control over these Coleman functions. This
allows us to bound their number of zeros and hence the size
of~$\cY(\bZ_S)$. The kind of control we are looking for is a bound on the
\emph{weight} of the Coleman functions, a notion introduced by Betts \cite{betts:effective}. 

To state our main results in this direction, we consider the decomposition of the $S$-integral points
\[ \cY(\bZ_S) = \coprod_{\Sigma} \cY(\bZ_S)_\Sigma \]
into points of given \emph{reduction types}~$\Sigma$
(see \cite[§6.2]{betts:effective}),
which control the mod-$\ell$ reduction for all primes~$\ell$. The refined Chabauty--Kim loci $\cY(\bZ_p)_{S,n}$ are similarly a union over reduction types
\[ \cY(\bZ_p)_{S,n} = \bigcup_{\Sigma} \cY(\bZ_p)_{S,n,\Sigma}, \]
with $\cY(\bZ_p)_{S,n,\Sigma}$ containing $\cY(\bZ_S)_\Sigma$.
For a prime $\ell$, we denote by 
$n_{\ell}$  
the number of irreducible components of the mod-$\ell$
special fibre of~$\cX$ (if $\ell \not\in S$), respectively of the minimal
regular normal crossings model of~$(X,D)$ over~$\bZ_{\ell}$ 
(if $\ell \in S$;
see~\cite[Appendix~B]{Bet18}).\footnote{The symbol $n_2$ has two different
	meanings: the number of conjugate pairs of complex points at infinity
	and the number of components of the mod-$2$ special fibre. 
	This should not cause any confusion as the correct meaning will always be clear from the context.}
Set $\kappa_p \coloneqq 1 + \frac{p-1}{(p-2) \log(p)}$ if~$p$ is odd and $\kappa_2 \coloneqq 2 + \frac{2}{\log(2)}$.

\begin{thmABC}\label{thm: beta}
	If 
	\[ \beta(Y, s, p) \coloneqq \frac12 g(g+3) - \frac12 r_p(r_p+3) + \rho_f +
	b -s > 0, \]
	then for each reduction type~$\Sigma$ there exists a nonzero Coleman algebraic function of weight at most~$2$ vanishing on~$\cY(\bZ_p)_{S,2,\Sigma}$. Moreover, the size of the refined Chabauty--Kim locus $\cY(\bZ_p)_{S,2}$ (and thus the number of $S$-integral points of~$\cY$) is bounded by
		\[ \# \cY(\bZ_p)_{S,2} \leq \kappa_p \cdot \prod_{\ell \in S} (n_{\ell} +
		n) \cdot \prod_{\ell \not\in S} n_{\ell} \cdot \# \cY(\bF_p)
		\cdot (4g+2n-2)^2 (g+1)\,. \]
\end{thmABC}

\begin{rem}[Weight~2 Coleman algebraic functions]
	\label{rem: weight 2 coleman functions}
	Coleman algebraic functions of weight at most~$2$ are the kind of functions
        showing up in quadratic Chabauty as in~\cite{BBM16, BD18, BD19,
        BD-II, BMS21}. They are linear combinations of double Coleman integrals,
        single Coleman integrals and rational functions.
        The precise form is given in~\cite[Lemma~4.1.13]{betts:effective}. 
      Coleman algebraic functions of general weight are constructed
      in~\cite[\S4.1]{betts:effective}; they form a subring of the algebra
      of all Coleman (analytic) functions defined by Besser~\cite{Bes02}. 
	The theory of Betts yields bounds for the number of zeros of Coleman
	algebraic functions of bounded weight. In this way, the bound on $\#\cY(\bZ_p)_{S,2}$
	in~\Cref{thm: beta} follows from the existence of the weight~2
        Coleman algebraic functions. Namely, each $\#\cY(\bZ_p)_{S,2,\Sigma}$ is bounded by $\kappa_p \cdot \# \cY(\bF_p)
        \cdot (4g+2n-2)^2 (g+1)$, and multiplying this by $\prod_{\ell \in S} (n_{\ell} +
        n) \cdot \prod_{\ell \not\in S} n_{\ell}$, the number of reduction types~$\Sigma$, yields the bound for $\#\cY(\bZ_p)_{S,2}$.
\end{rem}

\begin{rem}
	\label{rem: depth 1 bound}
	Using essentially the same  argument as in  our proof of 
        \Cref{thm: beta} we can
        show an analogous statement for the \emph{linear} Chabauty--Kim locus $\cY(\bZ_p)_{S,1}$. Namely, if the stronger condition 
	$\beta(Y,s,p) - \rho_f > 0$ 
	holds, then there are nonzero Coleman algebraic functions of weight
        at most~2 vanishing already on the $\cY(\bZ_p)_{S,1,\Sigma}$, and
        the upper bound on $\#\cY(\bZ_p)_{S,2}$ from \Cref{thm: beta}  already
        holds for $\#\cY(\bZ_p)_{S,1}$.
\end{rem}

Our main tool is Betts' theory of weight filtrations on (refined)
Selmer schemes introduced in~\cite{betts:effective}. This machinery reduces
statements about finiteness of and bounds for refined Chabauty--Kim loci to
calculations of local and global Bloch--Kato Selmer groups.
The general strategy
is reviewed in \Cref{sec: general strategy} below. At this point suffice it
to say that the theory takes as input a $G_{\bQ}$-equivariant quotient
$U_Y \twoheadrightarrow U$ of the $\bQ_p$-pro-unipotent étale fundamental
group~$U_Y$ of~$Y$ at an $S$-integral base point,  where we write
$G_K=\Gal(\overline{K}/K)$ for the absolute Galois
group of a field $K$. 
If one is able to
calculate or at least bound the dimensions of the Bloch--Kato Selmer groups of the
weight-graded pieces of~$U$, one gets finiteness of and bounds on 
the size of the associated refined Chabauty--Kim locus $ \cY(\bZ_p)_{S,U}$.
The sets $\cY(\bZ_p)_{S,n}$ above correspond to the choice $U_Y
\twoheadrightarrow U_{Y,n}$, the $n$-th quotient of $U_Y$ along the lower
central series. If one is willing to assume the Bloch--Kato conjecture, one
can choose for~$U$ the full fundamental group $U_{Y,\infty} = U_Y$.
The conditional estimates on the dimensions of the relevant Bloch--Kato
Selmer groups can then be used to obtain conditional bounds on the size of $\cY(\bZ_p)_{S,\infty}$ and hence on the number of $S$-integral points of~$\cY$. This is one of the main results of Betts--Corwin--Leonhardt \cite[Theorem 1.4]{BCL:effective}. In contrast with this, all our results are unconditional. 

In order to achieve this, we work with rather small quotients of the
fundamental group whose Bloch--Kato Selmer groups we are able to compute. For \Cref{thm: alpha}\ref{thm: depth 1 finiteness} we choose the abelianisation $U_{Y,1} = U_Y^{\ab}$ of $U_Y$; the relevant calculations are done in \Cref{sec: depth 1}. For Theorems~\ref{thm: alpha}\ref{thm: depth 2 finiteness} and~\ref{thm: beta}, rather than working with all of $U_{Y,2}$, we construct a certain intermediate quotient 
\[ U_{Y,2} \twoheadrightarrow U \twoheadrightarrow U_{Y,1}. \]
The construction of this intermediate quotient, which is carried out in
\Cref{sec: depth 2} below, is motivated by the analogous construction in
the projective case given by Balakrishnan--Dogra
in~\cite[Proposition~2.2]{BD-II} and generalises~\cite[Remark~3.3]{BD18}.

Finally, in \Cref{sec: weight geq -2 quotient} we also investigate the finiteness statements and bounds that we get from the weight~$\geq -2$ quotient of the fundamental group. This is another intermediate quotient between $U_{Y,2}$ and $U_{Y,1}$, which is in general larger than the one used for Theorems \ref{thm: alpha}\ref{thm: depth 2 finiteness} and~\ref{thm: beta}, so we expect to get finiteness and bounds under weaker conditions. The price to pay for this is that the conditions involve a term 
\begin{equation}\label{hBK}
 h_{\BK} \coloneqq \dim_{\bQ_p} \rH^1_f(G_{\bQ}, \Hom(\bigwedge^2 V_p
  \Jac_X, \bQ_p(1)))\,, 
\end{equation}
which we do not understand well but which is conjectured to vanish as a
consequence of the Bloch--Kato conjectures
\cite[Conjecture~2.8]{BCL:effective}:
\begin{conjecture}
	\label{conj: BK vanishing}
	$h_{\BK} = 0$.
\end{conjecture}

We are not assuming Conjecture~\ref{conj: BK vanishing} for our results but rather make the dependence on the conjecture explicit by including the term~$h_{\BK}$ in the statements. Our results obtained by working with the full weight~$\geq -2$ quotient of the fundamental group read as follows.

\begin{thmABC}
	\label{full weight -2 finiteness}
	If
	\[
	\gamma(Y,s,p) \coloneqq  g^2 - r_p + \rho + b -s - h_{\BK}>0,
	\]
	then~$\cY(\bZ_p)_{S,2}$ is finite.
\end{thmABC}
\begin{thmABC}
	\label{full weight -2 coleman function}
	If
	\[
	\delta(Y,s,p) \coloneqq \frac 12 g(3g+1) - \frac12 r_p(r_p+3) + \rho + b
	- s - h_{\BK}>0,
	\]	
	then for every reduction type~$\Sigma$ there exists a nonzero Coleman algebraic function of weight at most~2 that vanishes on $\cY(\bZ_p)_{S,2,\Sigma}$.
	Moreover, the size of the refined Chabauty--Kim locus $\#\cY(\bZ_p)_{S,2}$ (and thus the number of $S$-integral points of~$\cY$) is then bounded by the same bound as in \Cref{thm: beta}.
\end{thmABC}

Having stated our main results, let us conclude with a few remarks.

\begin{rem}[$r$ versus $r_p$]
	\label{rem: r versus rp}
	The $p^\infty$-Selmer rank $r_p$ and the Mordell--Weil rank~$r$ of the Jacobian
	of $X$ satisfy $r_p \geq r$, with equality if and only if the
	$p$-divisible part of the Tate--Shafarevich group of $\Jac_X$ is
	trivial, as predicted by the Tate--Shafarevich conjecture.
	One way
	to replace~$r_p$ with~$r$ in Theorems~\ref{thm: alpha} and \ref{full weight -2 finiteness} without assuming the conjecture
	is to modify the definition of the Selmer scheme using the
	``Balakrishnan--Dogra trick'' (see~\cite[Definition~2.2]{BD18} for
	the case of a projective curve, and~\cite[§6.3]{betts:effective}
	for the affine variant).
	The modified refined Chabauty--Kim loci $\cY(\bZ_p)^{\mathrm{BD}}_{S,n}$ are potentially smaller than $\cY(\bZ_p)_{S,n}$ but still contain $\cY(\bZ_S)$.
	The analogous finiteness statements are as follows:
	Let
	\begin{align*}
		\alpha'_1(Y,s)&\coloneqq g-r+b-s;\\
		\alpha'_2(Y,s)&\coloneqq \alpha'_1(Y,s)+\rho_f;\\
		\gamma'(Y,s,p)&\coloneqq g^2 - r + \rho + b -s - h_{\BK}.
	\end{align*}
	If $\alpha'_1(Y,s)>0$ (resp.\ $\alpha'_2(Y,s)>0$ or
        $\gamma'(Y,s,p)>0$), then the Chabauty--Kim locus $\cY(\bZ_p)_{S,1}^{\mathrm{BD}}$ (resp.\ $\cY(\bZ_p)_{S,2}^{\mathrm{BD}}$) is finite.
        
    The modified loci $\cY(\bZ_p)^{\mathrm{BD}}_{S,n}$ can also be written as a union of $\cY(\bZ_p)^{\mathrm{BD}}_{S,n,\Sigma}$ over all reduction types, and the analogues of Theorems~\ref{thm: beta} and~\ref{full weight -2 coleman function} are: Let
    \begin{align*}
    	\beta'(Y,s) &\coloneqq \frac12 g(g+3) - \frac12 r(r+3) + \rho_f +
    b -s;\\
    	\delta'(Y,s,p) &\coloneqq \frac 12 g(3g+1) - \frac12 r(r+3) + \rho + b
    	- s - h_{\BK}.
   \end{align*}
	If $\beta'(Y,s) > 0$ or $\delta'(Y,s,p) > 0$, then for every
        reduction type~$\Sigma$ there exists a nonzero Coleman algebraic
        function of weight at most~2 that vanishes on
        $\cY(\bZ_p)^{\mathrm{BD}}_{S,2,\Sigma}$, and the size of the locus
        $\cY(\bZ_p)^{\mathrm{BD}}_{S,2}$ is bounded by the same bound as in
        \Cref{thm: beta}. Note that this is also a bound for the
        number of $S$-integral points thanks to the inclusion $\cY(\bZ_S)
        \subseteq \cY(\bZ_p)^{\mathrm{BD}}_{S,2}$. As in~Remark~\ref{rem:
        depth 1 bound}, we also get the same bound on 
        $\#\cY(\bZ_p)^{\mathrm{BD}}_{S,1}$ 
        when $\beta'(Y,s) -\rho_f>0$.
\end{rem}

\begin{rem}[Dependence on $p$]
	\label{rem: dependence on p}
	The conditions in Theorems~\ref{thm: alpha}--\ref{full weight -2 coleman function} depend on~$p$ only through $r_p$ and $h_{\BK}$.  Therefore, if the Tate--Shafarevich conjecture and the Bloch--Kato conjecture \ref{conj: BK vanishing} are known to hold, then we have $r_p = r$ and $h_{\BK} = 0$, and the conditions are in fact independent of~$p$. As explained in \Cref{rem: r versus rp}, the dependence on $r_p$ can be avoided by using the Balakrishnan--Dogra trick, which is why $\alpha_1'(Y,s)$, $\alpha_2'(Y,s)$, and $\beta'(Y,s)$ are independent of~$p$. 
\end{rem}

\begin{rem}[Dependence on $\cY$ and $S$]
	\label{rem: dependence on S}
	Note that 
	the conditions in Theorems~\ref{thm: alpha}--\ref{full weight -2
		coleman function}
	do not depend on the choice of $S$-integral model $\cY/\bZ_S$,
	only on the generic fibre~$Y/\bQ$. They also do not depend on the
	set~$S$ but only on its cardinality $s = \#S$. The bounds on $\#\cY(\bZ_p)_{S,2}$ in Theorems~\ref{thm: beta} and~\ref{full weight -2 coleman function} do depend on $\cY/\bZ_S$ through the invariants $n_{\ell}$, i.e.\ the number of irreducible components of the special fibres.
\end{rem}

\begin{rem}
  It would be interesting to make the results of this paper explicit. 
  For a projective curve $X/\bQ$ satisfying $r=g$ and $\rho>1$
  explicit methods for the computation of $X(\bQ_p)^{\mathrm{BD}}_{U}$
  based on $p$-adic heights have
  been developed (and applied) in~\cite{BD18, BD-II, BDMTV19, BDMTV2}, where $U$ is the fundamental group quotient constructed in~\cite[\S3]{BD18}. 
  We
  expect that one could use similar methods to compute $\cY(\bZ_p)_{S,U}$
  (or at least a finite superset thereof) for the quotient $U$ constructed
  in Lemma~\ref{depth 2 quotient}, at least in some special cases.
\end{rem}
\subsection*{Acknowledgements}

We thank Netan Dogra for pointing out a mistake in an earlier version of this paper. We thank Alex Betts for helpful comments. We thank the referees for their suggestions.
The first author acknowledges support from the Deutsche Forschungsgemeinschaft (DFG, German Research Foundation) through TRR 326 Geometry and Arithmetic of Uniformized Structures, project number 444845124.
The second and third author were supported by an NWO Vidi grant.

\section{Examples}
\label{sec: examples}

In this section, we give some sample applications of our theorems. We keep
the notation of the previous section.
\begin{example}[Rank equals genus case]
	\label{rem: rank equals genus}
	Assume that $r_p = r = g$ and assume for simplicity that all points at infinity are
        rational. Then Theorems~\ref{thm: alpha} and \ref{thm: beta} simplify as follows: 
	\begin{enumerate}
		\item If $n - 1 - s > 0$, then $\cY(\bZ_p)_{S,1}$ is finite.
		\item If $n - 1 - s + \rho_f > 0$, then $\cY(\bZ_p)_{S,2}$ is finite, for every reduction type~$\Sigma$ there exists a nonzero Coleman algebraic function of weight at most~$2$ vanishing on~$\cY(\bZ_p)_{S,2,\Sigma}$, and $\#\cY(\bZ_p)_{S,2}$ is bounded as in \Cref{thm: beta}.
	\end{enumerate}
\end{example}

\begin{example}[Totally ramified superelliptic curves]
	\label{ex:hyp-odd}
        Let $Y/\bQ$ be an  affine superelliptic curve given by
        an equation $ y^m = f(x)$,
        where $m>1$, $f\in \bZ[x]$ is squarefree of degree $d>2$ and
        $\gcd(d,m)=1$.  Then we have $n=n_1=\#|D|=1$, so that $b=0$.
        Suppose that $r = g$. Then the Balakrishnan--Dogra variant (see \Cref{rem: r versus rp}) of \Cref{thm: alpha}\ref{thm: depth 2 finiteness} shows that
        $\cY(\bZ_p)^{\mathrm{BD}}_{\varnothing, 2}$ is finite. The variant of \Cref{thm: beta} shows that for every reduction type~$\Sigma$ there
        exists a nonzero Coleman algebraic function of weight at most~2
        vanishing on $\cY(\bZ_p)_{\varnothing,2,\Sigma}$ and that we have the simple bound
        \begin{equation}\label{eq:superbound}
	 \#\cY(\bZ) \leq \#\cY(\bZ_p)^{\mathrm{BD}}_{\varnothing,2} \leq \kappa_p \cdot \prod_{\ell} n_{\ell}\cdot
          \#\cY(\bF_p) \cdot 16g^2(g+1)\,. 
        \end{equation}
In particular, when $m=2$ and $d=2g+1\ge 3$, then $Y$ is an affine
  hyperelliptic curve of genus~$g$ and odd degree. 
  In this case, the finiteness of
  $\cY(\bZ_p)^{\mathrm{BD}}_{\varnothing, 2}$ was previously shown in \cite[Theorem~1.1]{BBM16}, and then again in \cite[Theorem~1.1]{BD18}. 
An upper bound for $\#\cY(\bZ)$ of order $g^3$ was given by
Balakrishnan--Dogra  in~\cite[Theorem~1.3]{BD19}, whereas 
the bound we obtain from
specialising~\Cref{eq:superbound} is of order~$g^4$ due to
Hasse--Weil.
Their bound is better because their construction is based on
differential operators that are special to hyperelliptic curves, whereas we
use the general approach of Betts. (Compare with the upper bound for
rational points in~\cite[Theorem~1.1]{BD19}; it is of order $g^3$ for
hyperelliptic curves, but of order $g^4$ for non-hyperelliptic curves.)
\end{example}
\begin{example}[Even degree hyperelliptic curves]\label{ex:hyp-even}
	Now let $Y/\bQ$ be an affine hyperelliptic curve given by an
        equation
        \[ y^2 = a_{2g+2}x^{2g+2} + \ldots + a_0, \quad a_i \in
        \bZ\,,\; a_{2g+2}\ne 0\,, \]
        where $a_{2g+2}x^{2g+2} + \ldots + a_0$ is squarefree. Then we have
        $n=2$.
        Suppose that
        $a_{2g+2}$ is either a square of an integer or is negative. Then $b=1$, and
        hence the Balakrishnan--Dogra variant of \Cref{thm: alpha}\ref{thm: depth 2 finiteness} gives an unconditional proof that $\cY(\bZ)$ is finite using
  non-abelian  Chabauty when $r\le g+1$. 
  Suppose in addition that $r=g$. Then $\alpha'_2(Y,s) = \beta'(Y,s) =
  \rho_f+1-s$, so 
        whenever $s < \rho_f+1$, we
  obtain the upper bound 
        \begin{equation*}\label{eq:evenbound}
          \#\cY(\bZ_S) \leq \#\cY(\bZ_p)^{\mathrm{BD}}_{S,2} \leq \kappa_p \cdot \prod_{\ell\in S}(n_{\ell}+2)\cdot \prod_{\ell\notin S} n_{\ell}\cdot
          \#\cY(\bF_p) \cdot (4g+2)^2(g+1)\,. 
        \end{equation*}
        In fact, when $r=g$, then $\alpha_1'(Y,0) >0$,
        and hence
\Cref{thm: alpha}\ref{thm: depth 1 finiteness} and \Cref{rem: r versus rp}
        imply that the depth~1 locus $\cY(\bZ_p)^{\mathrm{BD}}_{\varnothing,
        1}$ is already finite.
        In this case the conclusions of~\Cref{thm: beta} also hold for $\cY(\bZ_p)^{\mathrm{BD}}_{\varnothing,
        	1}$ by \Cref{rem: depth 1 bound}. In
        particular,  we get the  bound 
        \begin{equation*}\label{eq:depth1weight2}
          \# \cY(\bZ)  \le \# \cY(\bZ_p)^{\mathrm{BD}}_{\varnothing, 1}
                 \leq \kappa_p \cdot \prod_{\ell} n_{\ell} \cdot \# \cY(\bF_p)
		\cdot (4g+2)^2 (g+1)\,.
        \end{equation*}
\end{example}

\begin{example}[$S$-integral points on the thrice-punctured line]
	\label{ex: thrice-punctured line}
	Let $\cY \coloneqq \bP^1_{\bZ_S} \smallsetminus \{0,1,\infty\}$ be the thrice-punctured line. Assume that $2 \in S$; otherwise all refined Chabauty--Kim loci are automatically empty. We have $b = 2$ and $r_p = r = g = \rho_f = h_{\BK} = 0$ since the compactification $\bP^1$	has trivial Jacobian, so Theorems~\ref{thm: alpha}--\ref{full weight -2 coleman function} apply whenever $\#S < 2$, i.e., for $S = \{2\}$. Theorem~\ref{thm: alpha}\ref{thm: depth 2 finiteness} (or \Cref{full weight -2 finiteness}) yields the finiteness of $\cY(\bZ_p)_{\{2\},2}$, and Theorem~\ref{thm: beta} (or \Cref{full weight -2 coleman function}) shows that for each of the three reduction types~$\Sigma$ (corresponding to the three cusps $0$, $1$, $\infty$) there exists a nonzero Coleman algebraic function of weight at most~2 vanishing on $\cY(\bZ_p)_{\{2\},2,\Sigma}$, and we have the bound
	\[ \#\cY(\bZ[1/2]) \leq
	\#\cY(\bZ_p)_{\{2\},2} \leq 
	48(p-2) \kappa_p = 48 \Bigl(p-2 + \frac{p-1}{\log(p)}\Bigr). \]
	
	We actually know explicit equations cutting out the refined Chabauty--Kim loci $\cY(\bZ_p)_{S,2}$ for $\#S \leq 2$ by work of Best--Betts--Kumpitsch--Lüdtke--McAndrew--Qian--Studnia--Xu \cite{BBKLMQSX}. Namely, $\cY(\bZ_p)_{\{2\},2}$ consists of the common solutions in $\cY(\bZ_p)$ of the two equations
	\[ \log(z) = 0, \quad \Li_2(z) = 0, \]
	along with their translates under the natural $S_3$-action
	\cite[Theorem~A]{BBKLMQSX}. 
	The $p$-adic logarithm~$\log(z)$, which is Coleman algebraic of weight~$2$, is the function whose existence is predicted by \Cref{thm: beta} for one of the three reduction types. Indeed, it vanishes on~$\{-1\}$, the set of $\{2\}$-integral points reducing to the cusp $1$ modulo~$2$. The dilogarithm~$\Li_2(z)$ on the other hand is Coleman
	algebraic of weight~$4$, so its vanishing on $\{-1\}$ is not predicted by \Cref{thm: beta}. 
	The reason that the results of this paper do not capture all information coming from the depth~2 Chabauty--Kim locus $\cY(\bZ_p)_{S,2}$ is that $U_{Y,2}$ has a subgroup isomorphic to $\bQ_p(2)$, which is of weight~$-4$, whereas the proofs of our theorems make use only of the weight~$\geq -2$ quotient of $U_Y$.
	
	For the same reason, the results of this paper do not show the finiteness of $\cY(\bZ_p)_{S,2}$ for $\#S = 2$, although we know by \cite[Theorem~B]{BBKLMQSX} that the locus is finite also for such~$S$ and defined by explicit Coleman algebraic functions of weight~$4$.
\end{example}

\section{General strategy}
\label{sec: general strategy}

For the proofs of Theorems~\ref{thm: alpha}--\ref{full weight -2 coleman function} we follow Betts' strategy of exploiting weight filtrations on Selmer
schemes. Specifically, Theorem~\ref{thm: alpha} follows from a calculation of the dimensions of the
local and global Selmer scheme, and Theorem~\ref{thm: beta}
will follow from
\cite[Theorem~6.2.1\,A)+B)]{betts:effective}, for suitable quotients~$U_Y
\twoheadrightarrow U$ of the fundamental group. 

We briefly explain the strategy and sketch the arguments by which results in Chabauty--Kim theory follow from (abelian) Galois cohomology calculations. The reader who is willing to apply Betts' machinery as a black box may skip this section.

We write $G_v\coloneqq G_{\bQ_v}$ for a place $v$ of $\bQ$. Let $U$ be a $G_{\bQ}$-equivariant quotient
of~$U_Y$. The local Bloch--Kato Selmer scheme $\rH^1_f(G_p, U)$ consists of
the crystalline classes in $\rH^1(G_p, U)$.
Let $\rH^1_f(G_{\bQ}, U)$ denote the subspace of $\rH^1(G_{\bQ}, U)$ containing those classes that are crystalline at~$p$ and unramified at all other places.
Recall from~\cite[\S6.2]{betts:effective} that $\cY(\bZ_S)$ can be
partitioned according to reduction types. 
Two $S$-integral points have the same reduction type if and only if
for all primes $\ell$, their mod-$\ell$ reductions are either two
non-cuspidal points on the same irreducible component or are the same
cuspidal point. Here, mod-$\ell$ reduction refers to the special fibre of
$\cX$ (if $\ell \not\in S$), respectively of the minimal regular normal
crossings model of $(X,D)$ over $\bZ_{\ell}$ (if $\ell \in S$; see~\cite[Appendix~B]{Bet18}).

For every reduction type $\Sigma$, Betts--Dogra  define 
the global refined Selmer scheme $\Sel_{\Sigma,U}\subset \rH^1(G_\bQ, U)$
in~\cite[Definition~1.2.2]{BD:refined} (see also~\cite[\S3.2,
\S6.2]{betts:effective}).
Then we have a commutative diagram as follows~\cite[§6.2]{betts:effective}:
\[
\begin{tikzcd}
	\cY(\bZ_S)_{\Sigma} \dar["j_S"] \rar[hook] & \cY(\bZ_p) \dar["j_p"] \\
	\Sel_{\Sigma,U}(\bQ_p) \rar["\loc_p"] & \rH^1_f(G_p, U)(\bQ_p).
\end{tikzcd}
\]
The map $\loc_p$ in the bottom row is an algebraic map of affine $\bQ_p$-schemes.

\begin{defn}
	The \emph{refined Chabauty--Kim locus} $\cY(\bZ_p)_{S,U,\Sigma}$ for the reduction type~$\Sigma$ is the subset of those points of $\cY(\bZ_p)$ whose image in $\rH^1_f(G_p,U)$ is contained in the scheme-theoretic
	image of $\Sel_{\Sigma,U}$ under the localisation map. The total \emph{refined Chabauty--Kim locus} $\cY(\bZ_p)_{S,U}$ is defined as the union over all reduction types
	\[ \cY(\bZ_p)_{S,U} \coloneqq \bigcup_{\Sigma} \cY(\bZ_p)_{S,U,\Sigma}. \]
\end{defn}

We have the inclusion
$\cY(\bZ_S)_{\Sigma} \subseteq \cY(\bZ_p)_{S,U,\Sigma}$
for all reduction types~$\Sigma$, and hence
\[ \cY(\bZ_S) \subseteq \cY(\bZ_p)_{S,U}. \]
In particular, finiteness results and size bounds for the Chabauty--Kim loci imply the same for the set of $S$-integral points.

The set $\cY(\bZ_p)_{S,U}$ is finite whenever the localisation map~$\loc_p$ is not dominant, for example when
\[ \dim \Sel_{\Sigma,U} < \dim \rH^1_f(G_p, U) \]
for all~$\Sigma$. The dimensions of these non-abelian cohomology groups can be controlled by calculating dimensions of abelian cohomology groups arising as graded pieces of the weight filtration as follows. 

The pro-unipotent group~$U$ carries a weight filtration by subgroups \cite[Lemma~2.1.8]{betts:effective}
\[ \ldots \subseteq W_{-2} U \subseteq W_{-1} U = U, \]
such that $[W_{-i} U, W_{-j} U] \subseteq W_{-(i+j)} U$
for all $i,j \geq 1$. The graded pieces~$\gr_{-k}^W U = W_{-k} U / W_{-k-1}
U$ are representations of~$G_{\bQ}$ on finite-dimensional $\bQ_p$-vector spaces. 

\begin{lemma}
	\label{lem: Selmer scheme dimensions}
	Let $U_Y \twoheadrightarrow U$ be a finite-dimensional $G_{\bQ}$-equivariant quotient. Then the dimensions of the local and global Selmer scheme satisfy
	\begin{align*}
		\dim \Sel_{\Sigma,U} &\leq s + \sum_{k=1}^\infty \dim
                \rH^1_f(G_{\bQ}, \gr_{-k}^W U), \\
		\dim \rH^1_f(G_p, U) &= \sum_{k=1}^\infty \dim \rH^1_f(G_p,
                \gr_{-k}^W U). 
	\end{align*}
	In particular, if
	\begin{equation} 
		\label{eq: s-inequality}
		\sum_{k=1}^\infty \bigl( \dim \rH^1_f(G_p, \gr_{-k}^W U) -
                \dim \rH^1_f(G_{\bQ}, \gr_{-k}^W U) \bigr) - s > 0,
	\end{equation}
	then $\cY(\bZ_p)_{S,U}$ is finite.
\end{lemma}

\begin{proof}
	By~\cite[Lemma~3.2.6]{betts:effective}, each of the spaces $\Sel_{\Sigma,U}$ is
        (non-canonically) a closed subscheme of $\prod_{k=1}^\infty
        \rH^1_f(G_{\bQ}, \gr_{-k}^W U) \times \prod_{\ell \neq p}
        \mathfrak{S}_{\Sigma_{\ell}}$, where each $\fS_{\Sigma_{\ell}}$ is
        empty, a single point, or a curve of genus~0
        \cite[Lemma~6.1.4]{betts:effective}. 
        (Here, the vector spaces $\rH^1_f(G_{\bQ}, \gr_{-k}^W U)$ are viewed as affine schemes over~$\bQ_p$.) The latter can happen only for $\ell \in S$ by \cite[Lemma~6.1.1]{betts:effective}. 
        This implies the upper bound
        on~$\dim \Sel_{\Sigma,U}$. By
        \cite[Corollary~3.1.11]{betts:effective}, the local Selmer scheme
        $\rH^1_f(G_p, U)$ is (non-canonically) isomorphic to
        $\prod_{k=1}^\infty \rH^1_f(G_p, \gr_{-k}^W U)$, which implies the
        claim on its dimension. 

        If~\eqref{eq: s-inequality} is
        satisfied, then $\Sel_{\Sigma,U}$ has strictly smaller dimension
        than $\rH^1_f(G_p, U)$ for every reduction type~$\Sigma$. The
        localisation map is thus not dominant, which implies the finiteness
        of~$\cY(\bZ_p)_{S,U}$.
\end{proof}

Part \ref{thm: depth 1 finiteness} of \Cref{thm: alpha} will follow from \Cref{lem: Selmer scheme
dimensions} applied to the quotient $U_Y^{\ab}$ of~$U_Y$. For
part~\ref{thm: depth 2 finiteness} we will construct an intermediate quotient of $U_{Y,2} \twoheadrightarrow U_{Y,1} = U_Y^{\ab}$. 

In order to guarantee the existence of a Coleman algebraic function of a
certain weight vanishing on $S$-integral points, as in
Theorem~\ref{thm: beta}, Betts defines the following Hilbert series in
$\bN_0[\![t]\!]$ associated to~$U$:
\begin{align*}
	\label{def: Hilbert series}
	\HS_{\glob}(t) &\coloneqq (1-t^2)^{-s} \prod_{k=1}^\infty
        (1-t^k)^{- \dim \rH^1_f(G_{\bQ}, \gr_{-k}^W U)},\\
	\HS_{\loc}(t) &\coloneqq \prod_{k=1}^\infty (1-t^k)^{- \dim
        \rH^1_f(G_p, \gr_{-k}^W U)}.
\end{align*}

Denote their coefficients by $c_i^{\glob}$ and $c_i^{\loc}$, respectively.
The weight filtration on $U$ by subgroups induces a weight filtration on the $\bQ_p$-algebra $\cO(U)$ \cite[Example~2.1.5]{betts:effective}, which in turn induces weight filtrations on the rings of functions of the global and local Selmer scheme. The coefficients $c_i^{\glob}$ and $c_i^{\loc}$ are upper bounds for, respectively are equal to, the dimensions of the weight-graded pieces of the latter:

\begin{lemma}
	\label{lem: hilbert series coefficients}
	For all $i \geq 0$ we have
	\begin{align*}
		\dim \gr_i^W \cO(\Sel_{\Sigma,U}) &\leq c_i^{\glob}, \\
		\dim \gr_i^W \cO(\rH^1_f(G_p, U)) &= c_i^{\loc}.
	\end{align*}
\end{lemma}

\begin{proof}[Proof (sketch)]
	For every affine scheme $X/\bQ_p$ whose ring of functions is equipped with a weight filtration $W_\bullet$ we can define its Hilbert series as the generating function of the dimensions of the weight-graded pieces:
	\[ \HS_X(t) \coloneqq \sum_{i=0}^\infty \dim (\gr_i^W \cO(X)) \, t^i \in \bN_0^\infty[\![t]\!]. \]
	The claim is thus equivalent to
	\begin{align*}
		\HS_{\Sel_{\Sigma,U}}(t) \preceq \HS_{\glob}(t),\\
		\HS_{\rH^1_f(G_p, U)}(t) = \HS_{\loc}(t),
	\end{align*}
	where $\preceq$ denotes coefficient-wise inequality.

	In the case where $X = V = \prod_{k=1}^\infty V_{-k}$ is a weight-graded $\bQ_p$-vector space, viewed as an affine $\bQ_p$-scheme, there is an induced weight filtration on $\cO(V) = \mathrm{Sym}(V)^\vee$, and the Hilbert series is given by
	\[ \HS_V(t) = \prod_{k > 0} \HS_{V_{-k}}(t) = \prod_{k > 0} (1-t^k)^{- \dim V_{-k}}. \]
	As in the proof of \Cref{lem: Selmer scheme dimensions}, the local Selmer scheme $\rH^1_f(G_p, U)$ is non-canonically isomorphic to $\prod_{k=1}^\infty \rH^1_f(G_p, \gr_{-k}^W U)$, compatibly with the weight filtrations on rings of functions if the $k$-th factor is placed in weight~$-k$. This implies the  equality of their Hilbert series, the latter being precisely~$\HS_{\loc}(t)$.
	
	The global Selmer scheme $\Sel_{\Sigma,U}$ is non-canonically a closed subscheme of $\prod_{k=1}^\infty
	\rH^1_f(G_{\bQ}, \gr_{-k}^W U) \times \prod_{\ell \neq p}
	\mathfrak{S}_{\Sigma_{\ell}}$ as above. This implies the inequality of Hilbert series
	\[ \HS_{\Sel_{\Sigma,U}} \preceq \prod_{k=1}^\infty
	(1-t^k)^{- \dim \rH^1_f(G_{\bQ}, \gr_{-k}^W U)} \cdot \prod_{\ell \neq p} \HS_{\mathfrak{S}_{\Sigma_{\ell}}}(t). \]
	For $\ell \not\in S$ the scheme $\mathfrak{S}_{\Sigma_{\ell}}$ is either empty or a single point \cite[Lemma~6.1.1]{betts:effective}. For $\ell \in S$ we have the coefficient-wise bound \cite[Lemma~6.1.5]{betts:effective}
	\[ \HS_{\mathfrak{S}_{\Sigma_{\ell}}}(t) \preceq (1-t^2)^{-1}. \]
	The bound $\HS_{\Sel_{\Sigma,U}}(t) \preceq \HS_{\glob}(t)$ follows.
\end{proof}

\begin{lemma}[{\cite[Theorem~6.2.1\,A)]{betts:effective}}]
	\label{lem: coleman function from Hilbert series}
	Assume that $\sum_{i=0}^m c_i^{\glob} < \sum_{i=0}^m c_i^{\loc}$ for some positive integer~$m$. Then for each reduction type~$\Sigma$, there exists a nonzero Coleman algebraic function of weight at most~$m$ that vanishes on $\cY(\bZ_p)_{S,U,\Sigma}$.
\end{lemma}

\begin{proof}[Proof (sketch)]
	For each reduction type~$\Sigma$, the pullback map along the localisation map 
	\[ \loc_p^\sharp\colon \cO(\rH^1_f(G_p, U)) \to \cO(\Sel_{\Sigma,U}) \]
	is filtered with respect to the weight filtrations on both rings, i.e., maps functions of weight at most~$m$ to functions of weight at most~$m$. The spaces of functions of bounded weight are finite-dimensional. If we have the inequality
	\begin{equation} 
		\label{eq: weight inequality}
		\dim W_m \cO(\Sel_{\Sigma,U}) < \dim W_m \cO(\rH^1_f(G_p, U)),
	\end{equation}
	there is a nonzero element~$f$ of $W_m \cO(\rH^1_f(G_p, U))$ such that $f \circ \loc_p = 0$, and then $f \circ j_p$ is a nonzero Coleman algebraic function on $\cY(\bZ_p)$ of weight at most~$m$ that vanishes on $\cY(\bZ_p)_{S,U,\Sigma}$. 
	
	By \Cref{lem: hilbert series coefficients} we have
	\begin{align*}
		\dim \gr_i^W \cO(\Sel_{\Sigma,U}) &\leq c_i^{\glob}, \\
		\dim \gr_i^W \cO(\rH^1_f(G_p, U)) &= c_i^{\loc}
	\end{align*} 
	for all $i \geq 0$. Hence the inequality~\eqref{eq: weight inequality} is satisfied whenever $\sum_{i=0}^m c_i^{\glob} < \sum_{i=0}^m c_i^{\loc}$.
\end{proof}

\section{Weight-graded pieces of \texorpdfstring{$U_Y^{\ab}$}{the abelianised fundamental group}}\label{sec: depth 1}

The inclusion $Y \hookrightarrow X$ induces a surjection of $\bQ_p$-pro-unipotent fundamental groups $U_Y \twoheadrightarrow U_X$ and hence a surjection on their abelianisations: $U_Y^{\ab} \twoheadrightarrow U_X^{\ab}$. The latter group is the rational $p$-adic Tate module of the Jacobian of~$X$:
\[ U_X^{\ab} = V_p \Jac_X = \bigl(\varprojlim_{k} \Jac_X[p^k](\Qbar)\bigr) \otimes_{\bZ_p} \bQ_p. \]
The inclusions of the cusps induce a map 
\begin{equation*}
I\coloneqq \bQ_p(1)^{D(\Qbar)}/\bQ_p(1) \to U_Y^{\ab}\,,
\end{equation*}
where $\bQ_p(1)^{D(\Qbar)}$ denotes the Galois module of maps $D(\Qbar) \to \bQ_p(1)$, and $\bQ_p(1)$ is embedded diagonally, i.e.\ as the constant maps.
This yields a short exact sequence
\[ 1 \lto \underbrace{I}_{\gr_{-2}^W U_Y^{\ab}} \lto U_Y^{\ab} \lto \underbrace{V_p \Jac_X}_{\gr_{-1}^W U_Y^{\ab}} \lto 1, \]
which exhibits $U_Y^{\ab}$ as an extension of $V_p \Jac_X$ in weight~$-1$
by $I$ in weight~$-2$. We now calculate the
dimensions of the global and local Galois cohomology groups of these
weight-graded pieces.

First, the Galois cohomology of the Tate module follows from the work of
Bloch--Kato~\cite[Section~3]{BK90}, see~\cite[Theorem 3.11,
Corollary~3.12]{Cor21}:
\begin{align}
  \dim \rH^1_f(G_{\bQ}, V_p \Jac_X) &= r_p, \label{item: Tate module global}\\ 
                \dim \rH^1_f(G_p, V_p \Jac_X) &= g.  \label{item: Tate module local}
\end{align}

We now compute the Galois cohomology of the cuspidal inertia $I$.
\begin{lemma}[Galois cohomology of cuspidal inertia]
	\label{lem: Galois cohomology of cuspidal inertia}
	\leavevmode
	\begin{enumerate}[label=(\alph*)]
		\item $\dim \rH^1_f(G_{\bQ}, I) = n_1 + n_2 - \#|D|$,
		\label{item: inertia global}
		\item $\dim \rH^1_f(G_p, I) = n-1$.
		\label{item: inertia local}
	\end{enumerate}
\end{lemma}

\begin{proof}
	The sequence of Galois representations
	\[ 0 \lto \bQ_p(1) \lto \bQ_p(1)^{D(\Qbar)} \lto I \lto 0 \]
	is split exact: a retraction $\bQ_p(1)^{D(\Qbar)} \to \bQ_p(1)$ is given by the averaging map $(a_x)_{x \in D(\Qbar)} \mapsto \frac1{n} \sum_x a_x$. Hence it remains exact after applying $\rH^1_f(G_{\bQ}, -)$ or $\rH^1_f(G_p,-)$ and we have
	\begin{align*} 
		\dim \rH^1_f(G_{\bQ}, I) &= \dim \rH^1_f(G_{\bQ}, \bQ_p(1)^{D(\Qbar)}) - \dim \rH^1_f(G_{\bQ}, \bQ_p(1)),\\
		\dim \rH^1_f(G_p, I) &= \dim \rH^1_f(G_p, \bQ_p(1)^{D(\Qbar)}) - \dim \rH^1_f(G_p, \bQ_p(1)).
	\end{align*}
	
	To calculate the cohomology of~$\bQ_p(1)^{D(\Qbar)}$, note that $\bQ_p(1)^{D(\Qbar)} = (\pi_D)_* \bQ_p(1)$, where $\pi_D\colon D \to \Spec(\bQ)$ is the structural morphism. We get
	\[ \rH^1(G_{\bQ}, \bQ_p(1)^{D(\Qbar)}) = \rH^1(D, \bQ_p(1)) = \bigoplus_{x \in |D|} \rH^1(\kappa(x), \bQ_p(1)), \]
	where $\kappa(x)$ is the residue field of the cusp~$x$. This implies
	\[ \rH^1_f(G_{\bQ}, \bQ_p(1)^{D(\Qbar)}) = \bigoplus_{x \in |D|} \rH^1_f(\kappa(x), \bQ_p(1)), \]
	where $\rH^1_f(\kappa(x), \bQ_p(1))$ is the subspace of cohomology classes that are crystalline at all places of~$\kappa(x)$ dividing~$p$, and unramified at all other places. These are precisely the classes in the image of
	\[ \widehat{\cO_{\kappa(x)}^\times} \otimes \bQ_p \hookrightarrow \widehat{\kappa(x)^\times} \otimes \bQ_p = \rH^1(\kappa(x), \bQ_p(1)), \]
	where $\widehat{M} \coloneqq \varprojlim M/p^k M$ denotes the $p$-adic completion. By the Dirichlet Unit Theorem, we have
	\[ \dim_{\bQ_p} (\widehat{\cO_{\kappa(x)}^\times} \otimes \bQ_p) = r_1(x) + r_2(x) - 1 \]
	with $r_1(x)$ and $r_2(x)$ denoting the number of real embeddings, respectively pairs of complex embeddings of~$\kappa(x)$. Taking everything together we find
	\[ \dim_{\bQ_p} \rH^1_f(G_{\bQ}, \bQ_p(1)^{D(\Qbar)}) = \sum_{x \in |D|} (r_1(x) + r_2(x) -1) = n_1 + n_2 - \#|D|. \]
	Together with
	$\rH^1_f(G_{\bQ}, \bQ_p(1)) =0$,
	this implies~\ref{item: inertia global}. 
	
	For the local cohomology group, we calculate:
	\[
		\rH^1_f(G_p, \bQ_p(1)^{D(\Qbar)}) = \bigoplus_{x \in |D_{\bQ_p}|} \rH^1_f(\kappa(x), \bQ_p(1)) 
		= \bigoplus_{x \in |D_{\bQ_p}|} \widehat{\cO_{\kappa(x)}^\times} \otimes \bQ_p
	\]
	whose dimension is
	\begin{align*}
		\dim_{\bQ_p} \rH^1_f(G_p, \bQ_p(1)^{D(\Qbar)}) &= \sum_{x \in |D_{\bQ_p}|} \dim_{\bQ_p} \bigl(\widehat{\cO_{\kappa(x)}^\times} \otimes \bQ_p \bigr)\\
		&= \sum_{x \in |D_{\bQ_p}|} [\kappa(x) : \bQ_p] \\
		&= n.
	\end{align*}
	Since
	$\dim_{\bQ_p} \rH^1_f(G_p, \bQ_p(1)) =1$,
	this implies~\ref{item: inertia local}.
\end{proof}

\begin{rem}
	The proof of \Cref{lem: Galois cohomology of cuspidal inertia} simplifies if we assume that all cusps are rational: in this case $\bQ_p(1)^{D(\Qbar)}$ is a direct sum of copies of~$\bQ_p(1)$ rather than a twisted form of it.
\end{rem}

Having calculated the dimensions of Galois cohomology of the weight-graded pieces of $U_Y^{\ab}$, we can now prove part \ref{thm: depth 1 finiteness} of \Cref{thm: alpha}.

\begin{proof}[Proof of \Cref{thm: alpha}\ref{thm: depth 1 finiteness}]
	By \Cref{lem: Selmer scheme dimensions}, $\cY(\bZ_p)_{S,1}$ is finite whenever
	\begin{align*}
		0 &< \sum_{k=1}^2 \bigl( \dim_{\bQ_p} \rH^1_f(G_p,
                \gr^W_{-k} U_Y^{\ab}) - \rH^1_f(G_{\bQ}, \gr^W_{-k}
                U_Y^{\ab}) \bigr) - s \\
		&= (g - r_p) + ((n-1) - (n_1 + n_2 - \#|D|)) - s &
                \text{(\eqref{item: Tate module global},~\eqref{item: Tate module local} \& Lemma~\ref{lem: Galois cohomology of cuspidal inertia})} \\
		&= g - r_p + \#|D| + n_2 -1 -s. & \qedhere
	\end{align*}
\end{proof}

\begin{rem}
	Instead of working with the full abelianisation $U_Y^{\ab}$, we can form the pushout along the $G_{\bQ}$-equivariant map 
	\[ \bQ_p(1)^{D(\Qbar)}/\bQ_p(1) = \bigl( \bigoplus_{x \in |D|} \bQ_p(1)^{x(\Qbar)}\bigr) / \bQ_p(1) \twoheadrightarrow \bigl( \bigoplus_{x \in |D|} \bQ_p(1) \bigr)/\bQ_p(1) \]
	that takes the average on each Galois orbit of cusps. This
        produces an extension $U$ of $V_p \Jac_X$ by $\bQ_p(1)^{|D|} / \bQ_p(1)$. The cohomology of the latter is easier to calculate:
	\begin{align*}
		\dim \rH^1_f(G_{\bQ}, \bQ_p(1)^{|D|} / \bQ_p(1)) &= 0, \\
		\dim \rH^1_f(G_p, \bQ_p(1)^{|D|} / \bQ_p(1)) &= \#|D| - 1.
	\end{align*}
	The refined Chabauty--Kim locus $\cY(\bZ_p)_{S,U}$ associated to
        this quotient~$U$ is finite whenever $0 < g - r_p + \#|D| - 1 - s$. This is more restrictive compared to the full abelianisation~$U_Y^{\ab}$, where the right hand side contains an additional summand of~$n_2$. But if all cusps are real, i.e.\ if $n_2 = 0$, the conditions ensuring finiteness actually agree.
        
It would be interesting to study in more detail what kind of Coleman
functions appear in the equations defining the locus
$\cY(\bZ_p)_{S,U}$ (and similarly for the full abelianisation
$U^{\ab}_Y$).
In addition to the weight~1 functions arising from $U_X$, these are of
weight~2 by the above.
For instance, for the thrice-punctured line
        $\cY \coloneqq \bP^1_{\bZ[1/2]} \smallsetminus \{0,1,\infty\}$ 
        we have that
$\cY(\bZ_p)_{\{2\},1,\Sigma}$ is cut out by the weight 2 function
  $\log(z)$ for some reduction type $\Sigma$ by~\cite[Proposition~3.1]{BBKLMQSX}.
\end{rem}

\section{The abelian-by-Artin--Tate quotient}
\label{sec: depth 2}

The key step in the proofs of \Cref{thm: alpha}\ref{thm: depth 2 finiteness} and of \Cref{thm: beta} is the construction of a suitable quotient $U_Y \twoheadrightarrow U$ of the fundamental group of~$Y$, which is inspired by the proof of \cite[Proposition 2.2]{BD-II}. It lies between the depth~1 and depth~2 quotient,
\[ U_{Y,2} \twoheadrightarrow U \twoheadrightarrow U_{Y,1} = U_Y^{\ab}, \]
and can be described as the largest quotient of $U_{Y}$ of weight~$\geq -2$ which is a central extension of an abelian group by an Artin--Tate representation.\footnote{We call a $G_\bQ$-representation \emph{Artin--Tate} if it is a Tate twist of an Artin (i.e.\ finite image) representation.}

\begin{lemma}
	\label{depth 2 quotient}
	There exists a $G_{\bQ}$-equivariant quotient $U_Y
        \twoheadrightarrow U$ that is a central extension of $V_p \Jac_X$ by $(\bQ_p\otimes\NS(\Jac_{X_{\Qbar}}))^\vee(1) \oplus \bQ_p(1)^{D(\Qbar)} / \bQ_p(1)$. 
\end{lemma}

\begin{proof}
	We construct $U$ as a quotient of the weight~$\geq -2$ quotient $U_Y / W_{-3} U_Y$, which is a central extension as follows
	\begin{align}
		\label{eq: weight -2 extension}
		1 \lto \underbrace{\bigwedge^2 V_p \Jac_X \; \oplus \; \bQ_p(1)^{D(\Qbar)} / \bQ_p(1)}_{\gr^W_{-2} U_Y} \lto U_Y/W_{-3} U_Y \lto \underbrace{V_p \Jac_X}_{\gr^W_{-1} U_Y} \lto 1.
	\end{align}
	
	The proof of \cite[Lemma 2.10]{BCL:effective} identifies $\bQ_p\otimes\NS(\Jac_{X_{L}})$ with $\Hom_{G_{L}}(\bigwedge^2 V_p \Jac_X, \bQ_p(1))$ for any number field $L$ and thus 
	\begin{align} \label{eq: NS wedge2}
		\bQ_p\otimes\NS(\Jac_{X_{\Qbar}}) \cong \colim_L \Hom_{G_{L}}(\bigwedge^2 V_p \Jac_X, \bQ_p(1)) \subset \Hom(\bigwedge^2 V_p \Jac_X, \bQ_p(1)).
	\end{align}
	Therefore we have a $G_{\bQ}$-equivariant surjection $\bigwedge^2 V_p \Jac_X \twoheadrightarrow (\bQ_p\otimes\NS(\Jac_{X_{\Qbar}}))^\vee(1)$.
	Since the extension~\eqref{eq: weight -2 extension} is central, the kernel is normal in $U_Y/W_{-3} U_Y$ and we can form a pushout:
	\begin{equation}
	\label{diag: def depth 2 quot}
	\begin{tikzcd}[column sep = small]
		1 \rar & \gr_{-2}^W U_Y \dar[->>] \rar & U_Y / W_{-3} U_Y \dar[->>] \rar & \gr_{-1}^W U_Y \dar[equals] \rar & 1 \\
		1 \rar & (\bQ_p\otimes\NS(\Jac_{X_{\Qbar}}))^\vee(1) \oplus \bQ_p(1)^{D(\Qbar)} /
                \bQ_p(1) \rar & U \rar & V_p \Jac_X \rar & 1. 
	\end{tikzcd}
	\end{equation}
	The resulting quotient~$U$ of $U_Y$ is the desired extension.
\end{proof}

\begin{rem}
	\label{rem: geometric origin}
	As pointed out to us by Alex Betts, the quotient $U$ constructed in \Cref{depth 2 quotient} is of geometric origin, in that there exists a smooth connected variety $E/\bQ$ whose $\bQ_p$-pro-unipotent fundamental group is equal to $U$, and there is a morphism $f\colon Y \to E$ which induces the quotient map $f_*\colon U_Y \twoheadrightarrow U_E = U$. This variety can be constructed as a torsor $E \to J_X$ under a torus~$T$. The pullback of this torsor along the Abel--Jacobi map is trivial, giving rise to the morphism $f\colon Y \to E$. This construction generalises the $\bG_m^{\rho-1}$-torsor over $J_X$ of Edixhoven--Lido \cite{EL21} in their geometric quadratic Chabauty method. The generalisation is twofold: firstly, $T$ may be a non-split torus (so the fundamental group of~$T$ is Artin--Tate rather than Tate); secondly, our curve~$Y$, in contrast with the setting of Edixhoven--Lido, is affine rather than projective (so $T$ contains the toric part of the generalised Jacobian of~$Y$).
\end{rem}

The group $U$ from \Cref{depth 2 quotient} sits between $U_{Y,2}$ and $U_{Y,1}$ as follows:
\[ U_{Y,2} \twoheadrightarrow U_Y / W_{-3} U_Y \twoheadrightarrow U \twoheadrightarrow U_{Y,1}. \]
In particular, we have inclusions of Chabauty--Kim loci
\[ \cY(\bZ_p)_{S,2} \subseteq \cY(\bZ_p)_{S,U} \subseteq \cY(\bZ_p)_{S,1}. \]
In order to prove \Cref{thm: alpha}\ref{thm: depth 2 finiteness}
and~\Cref{thm: beta}, we need two preparatory lemmas that allow us to calculate the Selmer dimensions in weight $-2$.

\begin{lemma}
	\label{Selmer inf-res}
	\leavevmode
	\begin{enumerate}[label=(\alph*)]
		\item Let $K$ be a finite extension of $\bQ_\ell$ and $L/K$ a finite Galois extension with Galois group $G$.
		Let $V$ be a representation of $G_K$ on a finite dimensional $\bQ_p$-vector space ($\ell=p$ is allowed).
		Then the restriction is an isomorphism
		\[
			\rH^1_f(G_K,V) \cong \rH^1_f(G_L,V)^G.
		\]
		\label{item: inf-res local}
		\item Let $K$ be a number field and $L/K$ a finite Galois extension with Galois group $G$.
		Let $V$ be a representation of $G_K$ on a finite dimensional $\bQ_p$-vector space.
		Then the restriction is an isomorphism
		\[
		\rH^1_f(G_K,V) \cong \rH^1_f(G_L,V)^G.
		\]
		\label{item: inf-res global}
	\end{enumerate}
\end{lemma}

\begin{proof}
We use inflation--restriction, which also works for continuous group cohomology by \cite[Theorem 10.26]{othomas}.
Let us start with \ref{item: inf-res local}.
Then we have the exact sequence
\begin{align*}
0 \lto \rH^1(G,V^{G_L}) \lto \rH^1(G_K,V) \lto \rH^1(G_L,V)^G \lto \rH^2(G,V^{G_L}).
\end{align*}
Multiplication by $\# G$ is the zero map on $\rH^i(G, V^{G_L})$ for all $i
  \geq 1$, but it is also an isomorphism, since $V^{G_L}$ is a $\bQ_p$-vector space.
Thus restriction is an isomorphism
\[
	\rH^1(G_K,V) \cong \rH^1(G_L,V)^G.
\]

To see that this isomorphism restricts to an isomorphism on $\rH^1_f$, we distinguish the cases $\ell\neq p$ and $\ell=p$.
If $\ell\neq p$, an analogous inflation--restriction argument yields an isomorphism
\[
  \rH^1(I_K,V) \cong \rH^1(I_L,V)^{I_{L/K}},
\]
  where $I_K\subset G_K$, $I_L = I_K \cap G_L\subset G_L$ and $I_{L/K} \subset G$ are the inertia subgroups.
By the definition of $\rH^1_f$ for $\ell\neq p$, the rows in the following commutative diagram are exact:
\begin{equation}
\label{diag: 4 lemma}
\begin{tikzcd}[column sep = small]
0 \dar[equals] \rar & \rH^1_f(G_K,V) \dar \rar & \rH^1(G_K,V) \dar["\sim"',"\res"] \rar & \rH^1(I_K,V) \dar[hook] \\
0 \rar & \rH^1_f(G_L,V)^G \rar & \rH^1(G_L,V)^G \rar & \rH^1(I_L,V) . 
\end{tikzcd}
\end{equation}
We conclude by using the Four Lemma.
The case $\ell=p$ is similar, replacing the right vertical arrow in \eqref{diag: 4 lemma} by $\rH^1(G_K,V\otimes_{\bQ_p} \Bcris) \to \rH^1(G_L,V\otimes_{\bQ_p} \Bcris)$ and arguing as before.

For \ref{item: inf-res global}, inflation--restriction yields an isomorphism
\[
\rH^1(G_K,V) \cong \rH^1(G_L,V)^G
\]
in the same manner as above. 
Using the definition of the global Selmer groups, the rows in the following commutative diagram are exact:
\begin{equation}
\label{diag: 4 lemma global}
\begin{tikzcd}[column sep = small]
0 \dar[equals] \rar & \rH^1_f(G_K,V) \dar \rar & \rH^1(G_K,V) \dar["\sim"',"\res"] \rar & \prod_v \rH^1(G_{K_v},V)/\rH^1_f(G_{K_v},V) \dar["\prod_v\prod_{w|v} \res^{L_w}_{K_v}"] \\
0 \rar & \rH^1_f(G_L,V)^G \rar & \rH^1(G_L,V)^G \rar & \prod_v \prod_{w|v} \rH^1(G_{L_w},V)/\rH^1_f(G_{L_w},V),
\end{tikzcd}
\end{equation}
where $v$ runs over the finite places of $K$, 
and the horizontal arrows on the right are given by restricting to all decomposition groups.
Using part \ref{item: inf-res local} for every local Galois extension
  $L_w/K_v$, with $\Gal(L_w/K_v)\subset G$ as the decomposition group at
  $w$, we see that the rightmost vertical arrow in \eqref{diag: 4 lemma global} is injective.
So we apply the Four Lemma again to finish the proof.
\end{proof}

\begin{lemma}
	\label{Selmer dim AT}
	Let $W$ be a finite-dimensional $\bQ_p$-representation of $G_{\bQ}$ such that $W^\vee(1)$ is an Artin representation, i.\,e.\ has finite image.
	Then
	\leavevmode
	\begin{enumerate}[label=(\alph*)]
		\item $\dim \rH^1_f(G_p, W) = \dim W$,
		\label{item: AT Selmer local}
		\item $\dim \rH^1_f(G_{\bQ}, W) = -\dim \rH^0(G_{\bQ},W^\vee(1)) + \dim W - \dim W^{\sigma=1}$.
		\label{item: AT Selmer global}
	\end{enumerate}
\end{lemma}

\begin{proof}
	By assumption, there is a finite Galois extension $L/\bQ$ with Galois group $G$ such that $N\coloneqq W^\vee(1)$ restricted to $G_L$ is the trivial representation $\bQ_p^d$, where $d=\dim N = \dim W$.
	Note that $W = N^\vee(1)$ is equal to $\bQ_p(1)^d$ when restricted to $G_L$.
	
	For \ref{item: AT Selmer local}, let $L_p=L\bQ_p$ and $D_p = \Gal(L_p/\bQ_p) \subset G$.
	Then part \ref{item: inf-res local} of \Cref{Selmer inf-res} yields $\rH^1_f(G_p, W) \cong \rH^1_f(G_{L_p},\bQ_p(1)^d)^{D_p}$.
	By \cite[Prop.\ 2.9]{bellaicheBK}, the Kummer map gives an
        identification of  $\rH^1_f(G_{L_p},\bQ_p(1))$ with $\widehat{\cO^\times_{L_p}}\otimes_{\bZ_p}\bQ_p$ in a $D_p$-equivariant way, so \ref{item: AT Selmer local} follows from 
	\[
	\rH^1_f(G_{L_p},\bQ_p(1)^d)^{D_p} = (\widehat{\bZ_p^\times}\otimes_{\bZ_p}\bQ_p)^d = \bQ_p^d.
	\]
	
	For \ref{item: AT Selmer global}, we use the following consequence of Poitou--Tate duality \cite[Fact~2.9]{BCL:effective}:
	Let $W$ be a geometric Galois representation, i.e.\ $W$ is unramified at almost all places and de Rham at $p$. Then
	\begin{align}
	\label{eq: Poitou-Tate}
	\dim \rH^1_f(G_{\bQ}, W) &= \dim \rH^0(G_{\bQ}, W) + \dim \rH^1_f(G_{\bQ}, W^\vee (1))\\
	& \quad - \dim \rH^0(G_{\bQ}, W^\vee (1)) + \dim \rH^1_f(G_p, W) - \dim W^{\sigma=1}, \notag
	\end{align}
	where~$\sigma$ is complex conjugation (for some embedding $\Qbar \hookrightarrow \bC$).
	The given $W$ is geometric as $W^\vee(1)$ is Artin.
	The first summand $\dim \rH^0(G_{\bQ}, W)$ vanishes as $W$ is pure of weight $-2$.
	The representation $W^\vee(1)$ is trivial when restricted to $G_L$, thus part \ref{item: inf-res global} of \Cref{Selmer inf-res} yields $\rH^1_f(G_\bQ, W^\vee(1)) \cong \rH^1_f(G_L,\bQ_p^d)^{G}$.
	But $\rH^1_f(G_L,\bQ_p) = 0$ by \cite[Exercise 2.24.a]{bellaicheBK}. 
	Thus \eqref{eq: Poitou-Tate} and \ref{item: AT Selmer local} imply \ref{item: AT Selmer global}.
\end{proof}

\begin{proof}[Proof of \Cref{thm: alpha}\ref{thm: depth 2 finiteness}]
	Let $U_Y \twoheadrightarrow U$ be the quotient from \Cref{depth 2 quotient}. By construction, its weight-graded pieces are given by
	\begin{align*}
		\gr_{-1}^W U &= V_p \Jac_X, \\
		\gr_{-2}^W U &= (\bQ_p\otimes\NS(\Jac_{X_{\Qbar}}))^\vee(1) \oplus \bQ_p(1)^{D(\Qbar)} / \bQ_p(1).
	\end{align*}
	We calculate the dimensions of their Selmer groups.
In weight~$-1$, \eqref{item: Tate module global} and \eqref{item: Tate module local} yield
	\begin{align}
		\label{eq: dimension weight -1 global}
		\dim \rH^1_f(G_{\bQ}, \gr_{-1}^W U) &= r_p,\\
		\label{eq: dimension weight -1 local}
		\dim \rH^1_f(G_p, \gr_{-1}^W U) &= g.
	\end{align}	
In weight $-2$, we use \Cref{Selmer dim AT} with $W=(\bQ_p\otimes\NS(\Jac_{X_{\Qbar}}))^\vee(1)$.
Then $W^\vee(1) = \bQ_p\otimes\NS(\Jac_{X_{\Qbar}})$ is an Artin representation because $\NS(\Jac_{X_{\Qbar}})$ is a finitely generated abelian group.
Note that $W^\vee(1)^{G_\bQ} = \bQ_p\otimes \NS(\Jac_X)$ and $W^{\sigma=1}=(\bQ_p\otimes\NS(\Jac_{X_{\Qbar}})^{\sigma=-1})^\vee(1)$.
Together with Lemma~\ref{lem: Galois cohomology of cuspidal inertia} this yields
\begin{align}
\label{eq: dimension weight -2 global AT}
\dim \rH^1_f(G_{\bQ}, \gr_{-2}^W U) &= n_1 + n_2 - \#|D| + \dim W - \rho_f,\\
\label{eq: dimension weight -2 local AT}
\dim \rH^1_f(G_p, \gr_{-2}^W U) &= n - 1 + \dim W.
\end{align}
By~\Cref{lem: Selmer scheme dimensions}, $\cY(\bZ_p)_{S,U}$ and thus~$\cY(\bZ_p)_{S,2}$ are finite whenever	
	\begin{align*}
		0 &< \sum_{k=1}^2 \bigl( \dim_{\bQ_p} \rH^1_f(G_p,
                \gr^W_{-k} U) - \rH^1_f(G_{\bQ}, \gr^W_{-k} U) \bigr) - s\\
		&= (g - r_p) + ((n-1+\dim W) - (n_1 + n_2 - \#|D| + \dim W
                - \rho_f)) -s \\
		&= g - r_p + \rho_f + \#|D| + n_2 -1 -s. & \qedhere
	\end{align*}
\end{proof}

\begin{proof}[Proof of \Cref{thm: beta}]
	Again let $U_Y \twoheadrightarrow U$ be the quotient from
        \Cref{depth 2 quotient}. Having calculated the Galois cohomology
        dimensions of its weight-graded pieces in Eqs.~\eqref{eq: dimension
        weight -1 global}--\eqref{eq: dimension weight -2 local AT}, the
        global and local Hilbert series associated to~$U$ are given by
	\begin{align*}
		\HS_{\glob}(t) &= (1-t)^{-r_p} (1-t^2)^{-(s + n_1 + n_2 - \#|D| + \dim W - \rho_f)} \\
		&= 1 + r_p t + (s + n_1 + n_2 - \#|D| + \dim W - \rho_f + \tfrac12 r_p (r_p + 1)) t^2 + \ldots , \\
		\HS_{\loc}(t) &= (1-t)^{-g} (1-t^2)^{-(\dim W + n -1)} \\
		&= 1 + gt + (\dim W + n - 1 + \tfrac12 g(g+1))t^2 + \ldots 
	\end{align*}
	If the coefficients satisfy the inequality $\sum_{i=0}^2 c_i^{\glob} < \sum_{i=0}^2 c_i^{\loc}$, then \Cref{lem: coleman function from Hilbert series} applies and yields the existence, for every reduction type~$\Sigma$, of a nonzero Coleman algebraic function of weight at most~2 vanishing on~$\cY(\bZ_p)_{S,U,\Sigma}$ and thus on $\cY(\bZ_p)_{S,2,\Sigma}$. Abbreviating $d = \dim W$, we have:
	\begin{align*}
		& \sum_{i=0}^2 c_i^{\glob} < \sum_{i=0}^2 c_i^{\loc} \\
		\Leftrightarrow \; & 1 + r_p  + (s + n_1 + n_2 - \#|D| + d - \rho_f + \tfrac12 r_p (r_p + 1)) < 1 + g + (d + n - 1 + \tfrac12 g(g+1)) \\
		\Leftrightarrow \; & 0 < \tfrac12 g(g+3) - \tfrac12 r_p
                (r_p+3) + \rho_f + \#|D| + n_2 - 1 - s. 
	\end{align*}
        This proves the existence statement 
        of \Cref{thm: beta}.

Finally, the claimed bound on $\#\cY(\bZ_p)_{S,2}$ is
the one obtained from \cite[Theorem~6.2.1\,B)]{betts:effective} for $m = 2$, noting that the term $c_1^{\loc}$ appearing in the general formula is equal to $g$ by the calculation of the local Hilbert series above. Note that Betts' result is actually stated as a bound on $\#\cY(\bZ_S)$ but the proof, which goes via bounding the number of zeros of a Coleman algebraic function of bounded weight, applies in fact to the superset $\cY(\bZ_p)_{S,U}$.
\end{proof}

\begin{rem}
	As mentioned in \Cref{rem: depth 1 bound}, the same method can be
        used to show the depth~1 analogue of \Cref{thm: beta}. For this one
        simply replaces the quotient $U_Y \twoheadrightarrow U$ constructed
        in \Cref{depth 2 quotient} with the depth~1 quotient $U_{Y,1} =
        U_Y^{\ab}$.
\end{rem}

\begin{rem}
One can prove weaker versions of Theorems \ref{thm: alpha}\ref{thm: depth 2 finiteness} and \ref{thm: beta}, with $\rho$ in place of $\rho_f$ in the definition of $\alpha_2(Y,s,p)$ and $\beta(Y,s,p)$, by constructing a coarser quotient $U'$ of $U_Y / W_{-3} U_Y$.
Namely, the irreducible representation~$\bQ_p(1)$ occurs as a direct summand of the semisimple \cite[Lemma~6.0.1]{betts:effective} Galois representation $\bigwedge^2 V_p \Jac_X$ with multiplicity given by
\[ \dim_{\bQ_p} \Hom_{G_{\bQ}}(\bigwedge^2 V_p \Jac_X, \bQ_p(1)). \]
This dimension is equal to the Picard number~$\rho$ \cite[Proof of Lemma~2.10]{BCL:effective}.
Forming the pushout as in \eqref{diag: def depth 2 quot}, we obtain the quotient $U'$ of $U_Y / W_{-3} U_Y$ with $\gr_{-2}^W(U') = \bQ_p(1)^\rho \oplus \bQ_p(1)^{D(\Qbar)} / \bQ_p(1)$.
Its Selmer dimensions in weight $-2$ are therefore
\begin{align*}
\dim \rH^1_f(G_{\bQ}, \gr_{-2}^W U') &= n_1 + n_2 - \#|D|,\\
\dim \rH^1_f(G_p, \gr_{-2}^W U') &= \rho + n - 1.
\end{align*}
Now one proceeds as before to prove the analogues of \Cref{thm: alpha}\ref{thm: depth 2 finiteness} and~\Cref{thm: beta}.
\end{rem}
\begin{rem}
  Note that~\cite[Proposition~2.2]{BD-II}, which we used as our
  starting point for the results of the present section,
  strengthens~\cite[Lemma~3.2]{BD18}, which has $\rho$ in place of $\rho_f$.
  In fact, Balakrishnan--Dogra state in~\cite[Remark~3.3]{BD18} that one
  can use the same method as in the proof of~\cite[Lemma~3.2]{BD18} to show
  finiteness of $\cY(\bZ_p)_{\varnothing,2}^{\mathrm{BD}}$ when (in our
  notation) $n=\#D(\bQ)=1$ 
  and $g-r+\rho>0$, but give no further details.
  One may view \Cref{thm:
  alpha}\ref{thm: depth 2 finiteness} as a generalisation of this.
\end{rem}

\section{The full weight \texorpdfstring{$\geq -2$}{≥−2} quotient}
\label{sec: weight geq -2 quotient}

Let us consider the case where we choose for $U$ the full weight~$\geq -2$ quotient of the fundamental group:
\[ U \coloneqq U_Y/ W_{-3} U_Y. \]
As in~\eqref{eq: weight -2 extension} above, the graded piece of weight~$-2$ is semisimple and isomorphic to a direct sum
\[ \gr_{-2}^W U = \gr_{-2}^W U_Y \cong \bigwedge^2 V_p \Jac_X \; \oplus \; \bQ_p(1)^{D(\Qbar)}/\bQ_p(1). \]
The dimensions of the local and global Galois cohomology of
$\bQ_p(1)^{D(\Qbar)}/\bQ_p(1)$ have been calculated in \Cref{lem: Galois
cohomology of cuspidal inertia}, so we focus on the first summand. The
dimension of the global cohomology group~$\rH^1_f(G_{\bQ}, \bigwedge^2 V_p
\Jac_X)$ involves the term $h_{\BK}$, defined in~\eqref{hBK}.
\begin{lemma}[Galois cohomology of wedge-squared Tate module]
	\label{lem: Galois cohomology of wedge square}
	\leavevmode
	\begin{enumerate}[label=(\alph*)]
		\item $\dim \rH^1_f(G_{\bQ}, \bigwedge^2 V_p \Jac_X) = \tfrac12 g(g+1)- \rho + h_{\BK}$,
		\label{item: wedge-square global}
		\item $\dim \rH^1_f(G_p, \bigwedge^2 V_p \Jac_X) = \tfrac12 g (3g-1)$.
		\label{item: wedge-square local}
	\end{enumerate}
\end{lemma}

\begin{proof}
	We start with the local dimension. By \cite[Lemma~2.6]{BCL:effective}, the local Hilbert series of $Y$ is given by
	\[ \prod_{k=1}^\infty (1-t^k)^{- \dim \rH^1_f(G_p, \gr_{-k}^W U_Y)} = \frac{1-gt}{1-2gt-(n-1)t^2}. \]
	Expanding the power series up to the quadratic terms yields
	\[ 1 + d_1 t + (\tfrac12 d_1 (d_1+1) + d_2)t^2 + \ldots = 1 + gt + (2g^2 + n - 1)t^2+ \ldots, \]
	where $d_k \coloneqq \dim \rH^1_f(G_p, \gr_{-k}^W U_Y)$. Comparing coefficients yields 
	\[ d_1 = g, \quad d_2 = \tfrac12 g (3g-1) + n - 1.  \]
	Since $\gr^W_{-2} U_Y$ is a direct sum of $\bigwedge^2 V_p \Jac_X$ and $\bQ_p(1)^{D(\Qbar)}/\bQ_p(1)$, and the dimension of the local cohomology of the latter is $n-1$ by \Cref{lem: Galois cohomology of cuspidal inertia}, (b) follows.
	
	For the global cohomology, let $W \coloneqq \bigwedge^2 V_p \Jac_X$. We use \eqref{eq: Poitou-Tate} and go through the summands one by one. The first summand $\rH^0(G_{\bQ}, W)$ vanishes since~$W$ is pure of weight~$-2$. The second summand is precisely~$h_{\BK}$ since $W^\vee (1) = \Hom(W, \bQ_p(1))$. The third is
	\[ \dim \rH^0(G_{\bQ}, W^\vee (1)) = \dim \Hom(\bigwedge^2 V_p \Jac_X, \bQ_p(1))^{G_{\bQ}} = \rho, \]
	which we already used in the proof of \Cref{depth 2 quotient}. The
        local Galois cohomology $\rH^1_f(G_p, W)$ has dimension~$\tfrac12
        g(3g-1)$, as we just proved. Finally, consider $W$ as a
        representation of $G_{\infty} = \langle \sigma \rangle$. The two
        irreducible representations of $G_{\infty}$ 
        are the trivial representation~$\one$ and the sign representation, which we denote by~$\xi$. The isomorphism $V_p \Jac_X \cong (V_p \Jac_X)^\vee(1)$ given by the Weil pairing implies that the trivial representation and the sign representation appear in $V_p \Jac_X$ with equal multiplicity, so we have 
	\[ V_p \Jac_X \cong g\cdot\one \oplus g \cdot \xi \]
	as a $G_\infty$-representation, which implies
	\[ \bigwedge^2 V_p \Jac_X = g(g-1)\cdot \one \oplus g^2 \cdot \xi. \]
	In particular, 
	\begin{align}\label{eq:cc on wedge2}
		\dim W^{\sigma} = g(g-1).
	\end{align}
	Putting everything together in~\eqref{eq: Poitou-Tate} yields
	\begin{align*}
		\dim \rH^1_f(G_{\bQ}, \bigwedge^2 V_p \Jac_X) &= 0 + h_{\BK} - \rho + \tfrac12 g(3g-1) - g(g-1) \\
		&= \tfrac12 g(g+1) - \rho + h_{\BK},
	\end{align*}
	as claimed.
\end{proof}

Using these calculations, we can apply the general theory of Betts to the
quotient $U = U_Y/W_{-3} U_Y$ to obtain a finiteness criterion and a criterion for the existence of weight~2 Coleman functions vanishing on depth~2 Chabauty--Kim loci.

\begin{proof}[Proof of \Cref{full weight -2 finiteness}]
	The weight~$\geq -2$ quotient $U \coloneqq U_Y/W_{-3} U_Y$ is a
        quotient of $U_{Y,2}$, so we have $\cY(\bZ_p)_{S,2} \subseteq
        \cY(\bZ_p)_{S,U}$ and it suffices to show finiteness of the latter.
        The weight-graded pieces of $U$ are given by
	\begin{align*}
		\gr^W_{-1} U &= \gr^W_{-1} U_Y =  V_p \Jac_X,\\
		\gr^W_{-2} U &= \gr^W_{-2} U_Y = \bigwedge^2 V_p \Jac_X \; \oplus \; \bQ_p(1)^{D(\Qbar)}/\bQ_p(1).
	\end{align*}
	The dimensions of their local and global Galois cohomology follow
        from \eqref{item: Tate module global},~\eqref{item: Tate module
        local}, 
        Lemma~\ref{lem: Galois cohomology of cuspidal inertia} and Lemma~\ref{lem: Galois cohomology of wedge square}. In weight~$-1$ they are given by~$r_p$ (global) and $g$ (local), as above. In weight~$-2$ they are given by
	\begin{align}
		\label{eq: dimension full weight -2 global}
		\dim \rH^1_f(G_{\bQ}, \gr^W_{-2} U) &= \tfrac12 g(g+1) - \rho + h_{\BK} + n_1 + n_2 - \#|D|,\\
		\label{eq: dimension full weight -2 local}
		\dim \rH^1_f(G_p, \gr^W_{-2} U) &= \tfrac12 g(3g-1) + n -1.
	\end{align}
	By \Cref{lem: Selmer scheme dimensions}, the set $\cY(\bZ_p)_{S,U}$ is finite if
	\begin{align*}
		0 &< \sum_{k=1}^2 \bigl( \dim_{\bQ_p} \rH^1_f(G_p,
                \gr^W_{-k} U) - \rH^1_f(G_{\bQ}, \gr^W_{-k} U) \bigr) - s \\
		&= (g - r_p) + ((\tfrac12 g(3g-1) + n -1) - (\tfrac12
                g(g+1) - \rho + h_{\BK} + n_1 + n_2 - \#|D|)) - s \\
		&= g^2 - r_p + \rho + \#|D| + n_2 - 1 - h_{\BK} - s. \qedhere
	\end{align*}
\end{proof}

\begin{proof}[Proof of~\Cref{full weight -2 coleman function}]
	Choose $U \coloneqq U_Y/W_{-3} U_Y$ as above. Having calculated the dimensions of global and local cohomology of its weight-graded pieces in Eqs.~\eqref{eq: dimension weight -1 global}--\eqref{eq: dimension weight -1 local} and Eqs.~\eqref{eq: dimension full weight -2 global}--\eqref{eq: dimension full weight -2 local}, the associated Hilbert series can be calculated as follows:
	\begin{align*}
		\HS_{\glob}(t) &= (1-t)^{-r_p} (1-t^2)^{-(s + \tfrac12 g(g+1) - \rho + h_{\BK} + n_1 + n_2 - \#|D|)} \\
		&= 1 + r_p t + \bigl(\tfrac12 r_p(r_p+1) + s + \tfrac12 g(g+1) - \rho \\
		& \qquad + h_{\BK} + n_1 + n_2 - \#|D|\bigr)t^2 + \ldots,\\
		\HS_{\loc}(t) &= (1-t)^{-g} (1-t^2)^{-(\tfrac12 g(3g-1) + n -1)} \\
		&= 1 + gt + (\tfrac12 g(g+1) + \tfrac12 g(3g-1) + n -1)t^2 + \ldots \\
		&= 1 + gt + (2g^2 + n - 1)t^2 + \ldots
	\end{align*}
	Let $c_i^{\glob}$ and $c_i^{\loc}$ be the respective coefficients. By \Cref{lem: coleman function from Hilbert series}, for every reduction type~$\Sigma$ there exists a Coleman algebraic function of weight at most 2 that vanishes on $\cY(\bZ_p)_{S,U,\Sigma}$, whenever the following inequality holds:
	\begin{align*}
		&\sum_{i=0}^2 c_i^{\glob} < \sum_{i=0}^2 c_i^{\loc} \\
		\Leftrightarrow \; & 1 + r_p + \tfrac12 r_p(r_p+1) + s + \tfrac12 g(g+1) - \rho + h_{\BK} + n_1 + n_2 - \#|D| \\
		& \qquad < 1 + g + 2g^2 + n - 1 \\
		\Leftrightarrow \; & 0 < \tfrac 12 g(3g+1) - \tfrac12
                r_p(r_p+3) + \rho + \#|D| + n_2 -1 - s - h_{\BK}.
	\end{align*}
	Finally, the bound on $\#\cY(\bZ_p)_{S,2}$ from
        \cite[Theorem~6.2.1\,B)]{betts:effective} with $m=2$ depends on~$U$
        only through the Hilbert series coefficient $c_1^{\loc} = g$, so we
        get the same bound as in \Cref{thm: beta}. Here, as in the proof of \Cref{thm: beta} above, we are using the fact that Betts' bound on $\#\cY(\bZ_S)$ does in fact apply to the superset $\cY(\bZ_p)_{S,2}$.
\end{proof}


\printbibliography

\end{document}